\documentclass[11pt]{article}
\usepackage[french]{babel}
\usepackage{latexsym}
\usepackage{amsmath,amsxtra,amssymb,latexsym, amscd,amsthm,}
 \usepackage[pdftex]{color,graphicx}
 \usepackage[all]{xy}
\usepackage{pgf,tikz}
\usepackage[T1]{fontenc}
\usepackage{cleveref}
\usepackage{amsfonts}
\usetikzlibrary{arrows}
\usepackage{graphicx}
\usepackage{caption}
\usepackage{subcaption}
\usepackage{graphpap}
\usepackage{rotating}
\usepackage{bm}
\usepackage{fancyhdr}

\theoremstyle{plain}
\newtheorem{corollary}{Corollary}
\newtheorem{thm}{Theorem}
\newtheorem*{Thm}{Main theorem}
\newtheorem*{tHm}{Theorem}
\newtheorem{lemma}{Lemma}[section]

\newtheorem{Prop}{Proposition}[section]
\theoremstyle{definition}

\newtheorem{definition}{Definition}[section]

\newtheorem{remark}{Remark}[section]

\newcommand{\z}{\mathbb{Z}}

\newcommand{\sg}{\Sigma}

\title{\sc{Unicellular maps and filtrations of the mapping class group}}
\author{Abdoul Karim SANE}
\date{ \small{}}

\begin{document}
\renewcommand{\proofname}{Proof}
\renewcommand{\abstractname}{Abstract}
\renewcommand{\refname}{Bibliography}
\maketitle
\begin{abstract} 
This article first answers to questions about connectedness of a new family of graphs on unicellular maps. Answering these questions goes through a description of the mapping class group as surgeries on unicellular maps. We also show how unicellular maps encode subgroups of the mapping group and provide filtrations of the mapping class group. These facts add a layer on the ubiquitous character of unicellular maps.  
\end{abstract}

\begin{flushright}
 \tiny{\emph{Repositioning  "old molecules" \\ to draw subgroups of the mapping class group.}}
\end{flushright}

\begin{section}{Introduction}
Throughout this article, $\sg_g$ denotes a closed oriented surface of genus $g$ and $\mathcal{MCG}(\sg_g)$ the \textit{\textbf{mapping class group}} of $\sg_g$\string: the group of isotopy classes of preserving orientation homeomorphisms.

A \textit{\textbf{unicellular map}} is a graph $G:=(V,E)$ embedded in $\sg_g$ such that ~$\sg_g-G$ is a topological disk.  The \textit{\textbf{degree partition}} of a unicellular map $G$ is the ordered list $d:=(d_1,...,d_n)$ of degrees of the vertices of $G$.  
 We denote by ~$\widetilde{\mathcal{U}}_{d,g}$ the set of isotopy classes of unicellular maps on $\sg_g$ of degree partition $d$. 
 There is a natural action of the mapping class group on ~$\widetilde{\mathcal{U}}_{d,g}$\string:
 \begin{align*}
\mathcal{A}:\mathcal{MCG}(\sg_g)\times\widetilde{\mathcal{U}}_{d,g}&\longrightarrow \widetilde{\mathcal{U}}_{d,g}\\
                     (\phi,G)&\longmapsto \phi(G).
\end{align*}
We denote by $\mathcal{U}_{d,g}$ the quotient space under this action. This amounts to consider unicellular maps up to homeomorphisms.

In \cite{Element}, we introduced a new operation on unicellular maps called \textit{\textbf{surgery}}. A unicellular map comes with a natural cyclic order on its oriented edges. 
 Given a unicellular map $G$ and two oriented edges $x$ and $y$ of $G$, there is a unique simple arc $\lambda_{x,y}$ (up to homotopy with end points gliding on $x$ and ~$y$) leaving ~$x$ from its right side and entering $y$ into its right side. A surgery on ~$G$ between ~$x$ and ~$y$ consists of a local modification of ~$G$ along ~$\lambda_{x,y}$ as shown in Figure \ref{cutop}.

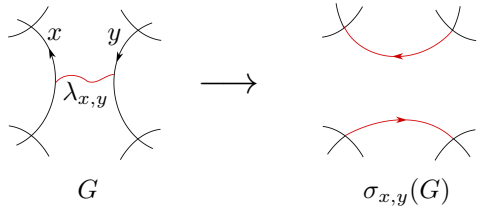
\begin{figure}[htbp]
\begin{center}
\[
   \xymatrix{
\includegraphics[scale=0.08]{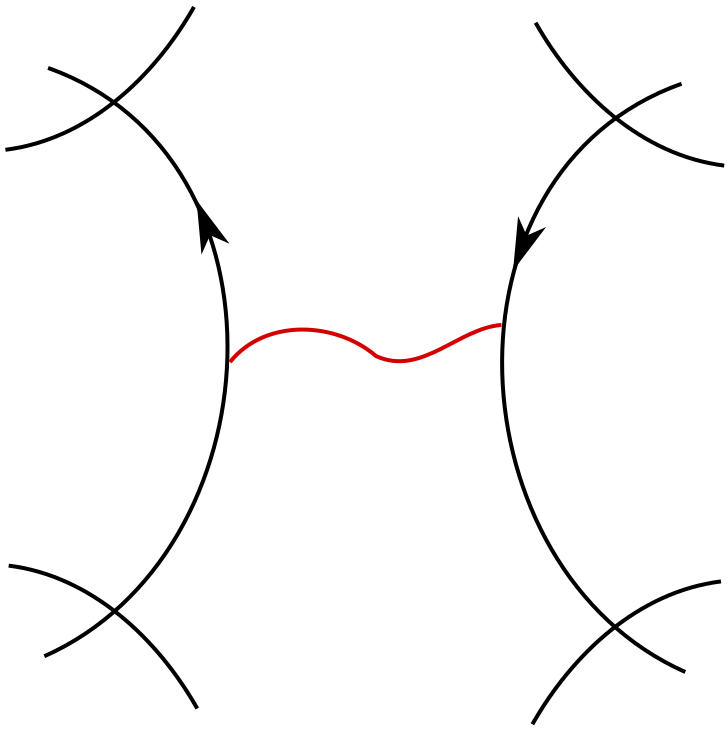}  &  &
\includegraphics[scale=0.08]{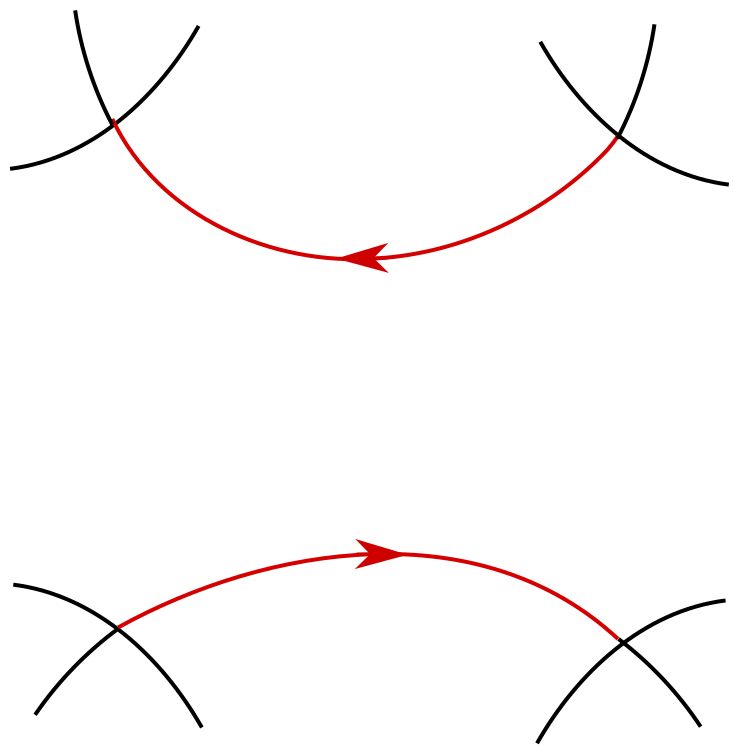}
\put(-57,16){\small{$x$}}
\put(-49,16){\small{$y$}}
\put(-55,8){\small{$\lambda_{x,y}$}}
\put(-37,9){\Large{$\longrightarrow$}}
\put(-53,-5){\small{$G$}}
\put(-15,-5){\small{$\sigma_{x,y}(G)$}}
}
\]
\caption{Surgery along ~$\lambda_{x,y}$.}
\label{cutop}
\end{center}
\end{figure}  
In \cite{Element} ---and we will recall it in Section \ref{sec2}---, we gave a necessary and sufficient condition (on how $x$ and $y$ appear in the cyclic order on edges)  for a surgery ~$\sigma_{x,y}(G)$ on $G$ between $x$ and $y$ to be a unicellular map. Since a surgery operation is between two oriented edges, the degree partition is left invariant.\vspace{0.2cm}

To the surgery operation, we associate two graphs\string:
\begin{itemize}
\item \textbf{\textit{the topological surgery graph $\widetilde{K}_{d,g}$}} whose vertices are elements of $\widetilde{\mathcal{U}}_{d,g}$, and whose edges are given by surgeries;
\item \textit{\textbf{the combinatorial surgery graph $K_{d,g}$}} whose vertices are elements of $\mathcal{U}_{d,g}$ with again surgeries as edges. 
\end{itemize}

The action of $\mathcal{MCG}(\sg_g)$ on $\widetilde{\mathcal{U}}_{d,g}$ extends to an action on $\widetilde{K}_{d,g}$ and the quotient graph is just $K_{d,g}$. Unlike the case of the other complexes (curves complex, pants complex...), combinatorial surgery graphs are big graphs since the number of homeomorphism classes of unicellular maps grows exponentially with the genus (see \cite{Chap}, \cite{Goup}, \cite{Walsh}).  

In \cite{Element}, we were interested on the coarse geometry of combinatorial surgery graphs ~$K_{d_1,g}$ for $d_1=\underbrace{(4,4,.....,4)}_{2g-1\hspace{0.1cm} \mathrm{times} }$. Since a regular 4-valent graph is equivalent to a collection of closed curves in generic position, we call elements of $\widetilde{\mathcal{U}}_{d_1,g}$ \textit{\textbf{unicellular collections}}. By  counting formulas due to Goupil and Schaeffer \cite{Goup}, we deduce that the number of unicellular collections is equivalent to $\displaystyle{\frac{(4g-2)!}{2^{2g-1}g!}}$. 

\begin{tHm}[S. \cite{Element}]\label{K4}
For every $g\geq1$, the combinatorial surgery graph ~$K_{d_1,g}$ is connected. Moreover, its diameter $D_g$ is zero for $g=1$ and satisfies the following inequality\string:
$$2g-1\leq D_g\leq3g^2+9g-12; $$
for $g\geq2$.
\end{tHm}

\noindent In this article, we first slightly extend this result to $K_{d_0,g}$; where $d_0=~(3,..,3)$.
\begin{thm}\label{K3}
For every $g\geq1$, the graph $K_{d_0,g}$ is connected. 
\end{thm}
\noindent Theorem \ref{K3} will be useful for the proof of\string:

\begin{Thm} For every $g\geq3$, the topological surgery graph $\widetilde{K}_{d_0,g}$ is connected. Moreover, $\widetilde{K}_{d_0,g}$ is quasi-isometric to $\mathcal{MCG}(\sg_g)$.
\end{Thm}

Let us explain the relation between topological surgery graphs and the mapping class group. 

Let $G$ be an element of $\widetilde{\mathcal{U}}_{d,g}$\string:

\begin{definition} An homeomorphism $\phi\in\mathcal{MCG}(\sg_g)$ is \textbf{\textit{$G$-surgery compatible}} if and only if there is a sequence of surgeries from $G$ to $\phi({G})$.  
\end{definition}

We denote by $\mathcal{MCG}_G(\sg_g)$ the group of all $G$-surgery compatible homeomorphisms of $\sg_g$. If $G$ and $G'$ are two unicellular maps in the same connected component of $\widetilde{K}_{d,g}$, then $\mathcal{MCG}_G(\sg_g)=\mathcal{MCG}_{G'}(\sg_g)$. When the combinatorial surgery graph $K_{d,g}$ is connected, $\mathcal{MCG}_G(\sg_g)$ up to conjugacy, depends only on the parameters $(d,g)$ and we just denote by $\mathcal{MCG}_{d,g}(\sg_g)$ the group of surgery compatible homeomorphisms.

If  $\widetilde{K}_{d,g}$ is connected, then $K_{d,g}$ is also connected and the group of surgery compatible homeomorphisms is the whole mapping class group; and it is an equivalence.

The proof of Main theorem relies on two facts\string: the connectivity of $K_{d_0,g}$ and the fact that $\mathcal{MCG}_{d_0,g}(\sg_g)=\mathcal{MCG}(\sg_g)$.\vspace{0,3cm} 

From our main theorem one can read a presentation of $\mathcal{MCG}(\sg_g)$. Let  ~$\gamma$ be closed path on $K_{d_0,g}$ based on $G\in K_{d_0,g}$;  $\gamma$ is a sequence of surgeries $G_0=G\rightarrow G_1\rightarrow...\rightarrow G_n$ such that $[G]=[G_n]$. It defines a unique homeomorphism (up to the group $\mathrm{Sym}(G)$  of homeomorphisms which fixe ~$G$) ~$\phi_{\gamma}$ such that  $\phi_{\gamma}(G_n)=G$. It follows from Main theorem that the map

 \begin{align*}
\mathcal{R}:\pi_1(K_{d_0,g}, G)&\longrightarrow \mathcal{MCG}(\sg_g)_{|\mathrm{Sym}(G)}\\
                     \gamma &\longmapsto \phi_{\gamma};
\end{align*} 
is surjective. This map gives a presentation of the mapping class group where the generating set corresponds to loop of surgeries on cubic unicellular maps with implicit relations coming from relations between surgeries. 

As we will see after the proof of Main theorem, if one just wants to generate the mapping class group via surgeries on a unicellular map, it is not necessary for the map to be cubique. There are other degree partitions $d$ for which the graph $K_{d,g}$ is disconnected and such that $\mathcal{MCG}_{G}(\sg_g)=\mathcal{MCG}(\sg_g)$ for a map ~$G\in~\widetilde{\mathcal{U}}_{d,g}$. There are also degree partitions for which the combinatorial surgery graph is connected but the group of surgery compatible homeomorphisms is a proper subgroup of ~$\mathcal{MCG}(\sg_g)$. Let us see this closely.  

Given a unicellular map $G$ and a vertex $v_i$ of $G$ of degree $d_i$, one can obtained a new unicellular maps by splitting $v_i$ into two vertices $v_j$ and $v_k$ of degree $d_j$ and $d_k$ and the following is satisfied\string: $d_i=d_j+d_k-2$.
\begin{definition}
A degree partition $d=(d_1,...,d_n)$ of a unicellular map is \textbf{\textit{even}} if all the integer $d_i$ are even.  
\end{definition}

We show the following\string:
\begin{thm}\label{inv} Let $G$ be a unicellular map with even degree partition $d$. Then $\mathcal{MCG}_G(\sg_g)$ is a proper subgroup of the mapping class group.   
\end{thm}  
Theorem ~\ref{inv} above implies that for even degree partitions, the topological surgery graph $\widetilde{K}_{d,g}$ is not connected and the graph $\widetilde{K}_{d_1,g}$ is one of them. We prove it by constructing an invariant of surgery on unicellular maps with even degree partition.

The splitting vertex operation associated to Main theorem give the following\string:
\begin{thm}\label{filt} For every $g\geq2$, there is a filtration of the mapping class group associated to each vertex splitting sequence starting at a one vertex unicellular map and ending to a cubic unicellular map. Moreover, for a suitable choice of the initial map $G$, one can start the filtration with a free group.    
\end{thm}

\begin{paragraph}{Comparaison with others graphs\string:}
One way to understand properties of the mapping class group  is by analyzing its shadow in a complex associated to ~$\sg_g$, on which it acts. The curves complex $C(\sg_g)$ is one of those complexes on which the mapping acts. Masur and Minsky (\cite{Minsk1},\cite{Minsk2}) showed that ~$C(\sg_g)$ is infinite diameter, $\delta$-hyperbolic and they use the action of $\mathcal{MCG}(\sg_g)$ on ~$C(\sg_g)$ and hyperbolic machineries to study properties like words problem, conjugacy problem, quasi-isometry rigidity. Masur and Minsky also showed that the mapping class group is hierarchically hyperbolic; a weaker version of hyperbolicity.  The arcs complex (a generalization of the curves complex) and the pants graph are other graphs on which the mapping class group acts naturally.

Previously, Hatcher and Thurston \cite{HaThu} used the graph of \textit{\textbf{cut systems}}\string: collections of disjoint closed curves which cut $\sg_g$ into a sphere with $2g$ boundary components, to read off a finite presentation of the mapping class group.

So, finding a complex on which the mapping class acts happens to be useful. Either we deduce from that action properties on the mapping class group or we deduce from that action properties on the complex. Our construction follows this philosophy. The graphs we provide are closer to the mapping class group than the others mention above since they are not quasi-isometric to the mapping (they are not locally compact).

The flip graph on one-vertex triangulations (see \cite{Hatch} for flips on triangulations) of $\sg_g$ is somehow closer to ~$\widetilde{K}_{d_0,g}$, since their vertices are dual each other. But, a surgery on a cubic unicellular map corresponds to a sequence of flips on the triangulation dual to it; and vice-versa. A surgery needs to satisfy an intertwining condition on edges while each edge of a triangulation corresponds to a flip. The degree of vertices in flip graphs are larger than those in surgery graphs, and connectedness is less expected for the latest graph. Its known \cite{DisHug} that the flip graph on triangulations with $n$ vertices is quasi-isometric to $\mathcal{MCG}(\sg_{g,n})$\string: the group of isotopy classes of homeomorphism which fixe the marked points (with fixed marked points isotopies). By Birman exact sequence, $\mathcal{MCG}(\sg_{g,1})$ is a $\pi_1(\sg_g)$-extension of $\mathcal{MCG}(\sg_g)$. So our main theorem provides a graph which is quasi-isometric to $\mathcal{MCG}(\sg_g)$.

For other degree partitions, unicellular maps are dual equivalent to one-vertex cell decompositions of the surface and there are also elementary move associated to some of them (see \cite{Nak1}, \cite{Nak2} for moves on quadrangulations). When we restrict to one-vertex cell decompositions, surgery (on the dual) seems to be the most adapted moves since it is defined for all degree partitions.      
\end{paragraph}
\begin{paragraph}{Outline of the paper\string:} In section \ref{sec2}, we recall some notions about surgery on unicellular maps and we give the proof of Theorem ~\ref{K3}. Section ~\ref{sec3} is divided into three paragraphs. The first one deals with the proof of Main theorem. In the second paragraph, we construct a surgery invariant for the proof of Theorem \ref{inv}. The third paragraph composed with examples in Section \ref{sec4} can be taken as a proof of Thereom ~\ref{filt}.   
\end{paragraph} 
\end{section}
\begin{section}{Surgery on unicellular maps}\label{sec2}
In this sections, we recall some notions about surgery on unicellular maps introduced in \cite{Element} by us. It ends by a
the proof of Theorem \ref{K3}.

Let $G:=(V,E)$ be a unicellular map on $\sg_g$ and let $d=(d_1,...,d_n)$ be its degree partition. In this section, we assume that all vertices have degree greater than two (even though the question of surgery on unicellular maps makes sense for unicellular maps having degree 1 and degree 2 vertices). In section \ref{sec3}, it will be more convenient to consider \textit{\textbf{false vertices}} which are just degree two vertices on each edges of $G$. 
Using the Euler characteristic, one obtained the following relation between the number of vertices, the number of edges and the degree partition of $G$\string:  
\[|E|-|V|=2g-1; \quad \displaystyle{\sum d_i=2|E|.} \]

\noindent For instance, if $G$ is a cubic unicellular map\string: $|V|=4g-2$ and $E=6g-3$. Cubic unicellular maps have largest number of vertices and edges.

Now, if one walk along a unicellular map $G$ in such a way that at each vertex, we follow the edge just on the right, we obtained a \textit{\textbf{closed walk}} and any oriented edges of $G$ is followed exactly once. The closed walk give a cyclic order on oriented  edges of $G$. If we label each oriented edge of ~$G$ such that two orientations of the same edge have the same label with the bar symbole ~( ~$\bar{}$ ~) over one of them, we obtained a word ~$W_{G}$ of length ~$2|E|$. Up to cyclic permutation and relabelling, the word $W_G$ determines the homeomorphism class of $G$.

Let $x$ and $y$ be two oriented edges of $G$. There is a unique arc $\lambda_{x,y}$ ---up to homotopy with extremities gliding in $x$ and $y$--- leaving $x$ from its right and entering $y$ into its left. Let $G':=\sigma_{x,y}(G)$ be the graph obtained by modifying ~$G$ along $\lambda_{x,y}$ like in Figure \ref{cutop}. 
\begin{definition}
The oriented edges $x$ and $y$ are \textit{\textbf{intertwined}} if and only if the arc $\lambda_{x,y}$ and $\lambda_{\bar{x},\bar{y}}$ intersect once. It is equivalent to say that in the cyclic order on  oriented edges we see $x<\bar{x}<y<\bar{y}$ or $x<\bar{y}<y<\bar{x}$.
\end{definition} 
In \cite{Element}, we proved the following\string: 
\begin{lemma}[Card Shuffling]\label{lemsurg}
The graph $G':=\sigma_{x,y}(G)$ is unicellular if and only if $x$ and $y$ are intertwined and in this case. 

\noindent Moreover, if $w_1\bm{x}w_2\bm{\bar{x}}w_3\bm{y}w_4\bm{\bar{y}}$ is a word associated to $G$, then $w_3\bm{x}w_2\bm{\bar{x}}w_1\bm{y}w_4\bm{\bar{y}}$ is a word associated to $G'$.
\end{lemma}

If $x$ and $y$ are two intertwined edges, so are $\bar{x}$ and $\bar{y}$; and by Lemma \ref{lemsurg}, there is an homeomorphism $\phi$ such that $\sigma_{\bar{x},\bar{y}}(G)=\phi(\sigma_{x,y}(G))$.

Now, we turn to the combinatorial surgery graph on cubic unicellular maps. Namely we prove Theorem \ref{K3}, and as we will see, it is a corollary of Theorem ~\ref{K4} already proved in \cite {Element}.

\begin{definition}
\begin{itemize}
\item An edge of a graph is a \textit{\textbf{bridge}} if its complement in the graph is disconnected. 
\item A \textit{\textbf{perfect matching}} of a graph is a maximal subset $P$ of edges such that any vertex of $G$ belong to an edge in $P$ and two different edges do not share a vertex. 
\end{itemize}
\end{definition}
The following theorem gives a sufficient condition for a cubic graph to admit a perfect matching.
\begin{thm}[Petersen, \cite{Pet}]\label{pet}
A cubic unicellular map with at most two bridges admits a perfect matching. 
\end{thm}
Starting from a unicellular collection $G$, one can get a cubic one $G'$ by splitting all his vertices into vertices of degree 3. Thus, $G'$ admits a perfect matching which is just the set of new edges obtained after the splitting process. But, not all cubic unicellular maps are obtained in this way.
\begin{definition} 
A cubic graph is a \textbf{\textit{virtual unicellular collection}} if it admits a perfect matching.  
\end{definition}
Virtual unicellular collections a closed to unicellular collection since they are obtained by a splitting vertex process. 

\begin{lemma} \label{perfmatch1}
Let $G$ be a cubic unicellular map. Then, there is a sequence $G_0=G\rightarrow G_1\rightarrow...\rightarrow G_n$ such that $G_n$ is a virtual unicellular collection. 
\end{lemma}

\begin{figure}[htbp]
\begin{center}
\includegraphics[scale=0.18]{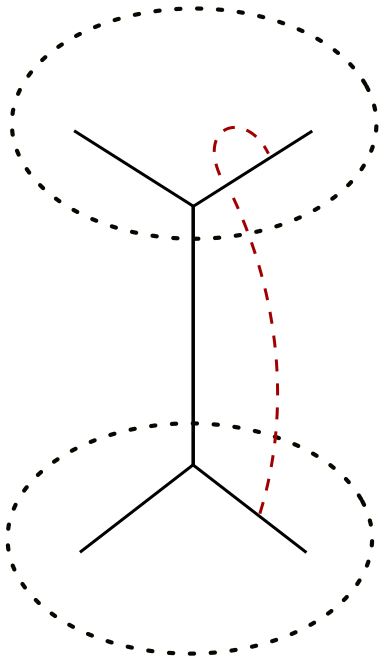}\hspace{3cm}
\includegraphics[scale=0.18]{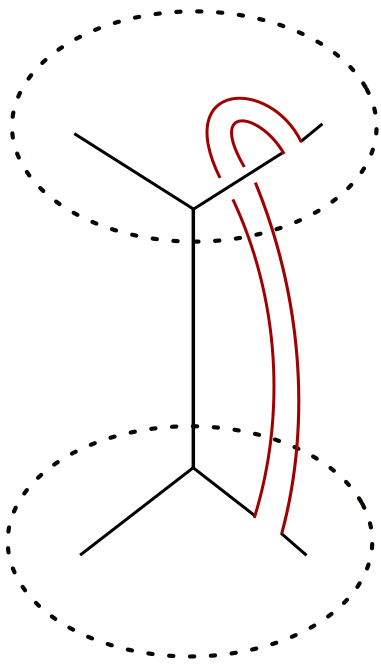}
\put(-70,20){\small{$e$}}
\put(-70,37){\small{$G_1(e)$}}
\put(-70,3){\small{$G_2(e)$}}
\put(-45,20){\huge{$\longrightarrow$}}
\caption{Reduction of the number of bridges by a surgery.}
\label{ponts}
\end{center}
\end{figure}

\begin{proof}
Let $k$ be the number of bridges in $G$. Let $e$ be a bridge of $G$, $G_1(e)$ and $G_2(e)$ the two components of $G-e$.

Let $e_1$ (respectively $e_2$) be one of the two edges of $G_1(e)$ (respectively $G_2(e)$) which has a common vertex with $e$. Since $e$ is a bridge, there is an orientation of $e_1$ ---let us denote it $x$--- which is intertwined with an orientation of $e_2$ ---which we denote $y$. 

The edge $e$ is no longer a bridge in $G_1:=\sigma_{x,y}(G)$ since the surgery produces another edges with extremities in $G_1$ and $G_2$ (see Figure ~\ref{ponts}). 

The number of bridges in $G_1$ is therefore equal two $k-1$. Repeating this process, we get a cubic unicellular map $G_n$ with no bridge; and then by Theorem \ref{pet},  $G_n$ is a virtual unicellular collection. 
\end{proof}
 
 Let us discuss the proof of Theorem \ref{K3} in genus $7$; for the general proof follows the same idea. Let $C_7$ be the unicellular collection in $\sg_7$ on Figure ~\ref{mitochondrie} (on the left), and $CV_7$ the cubic unicellular map obtained from ~$C_7$ by splitting all its vertices like on Figure \ref{mitochondrie} (on the right). The vertices in green are split in such a way that they give  bridge edges. 
\begin{figure}[htbp]
\begin{center}
\includegraphics[scale=0.22]{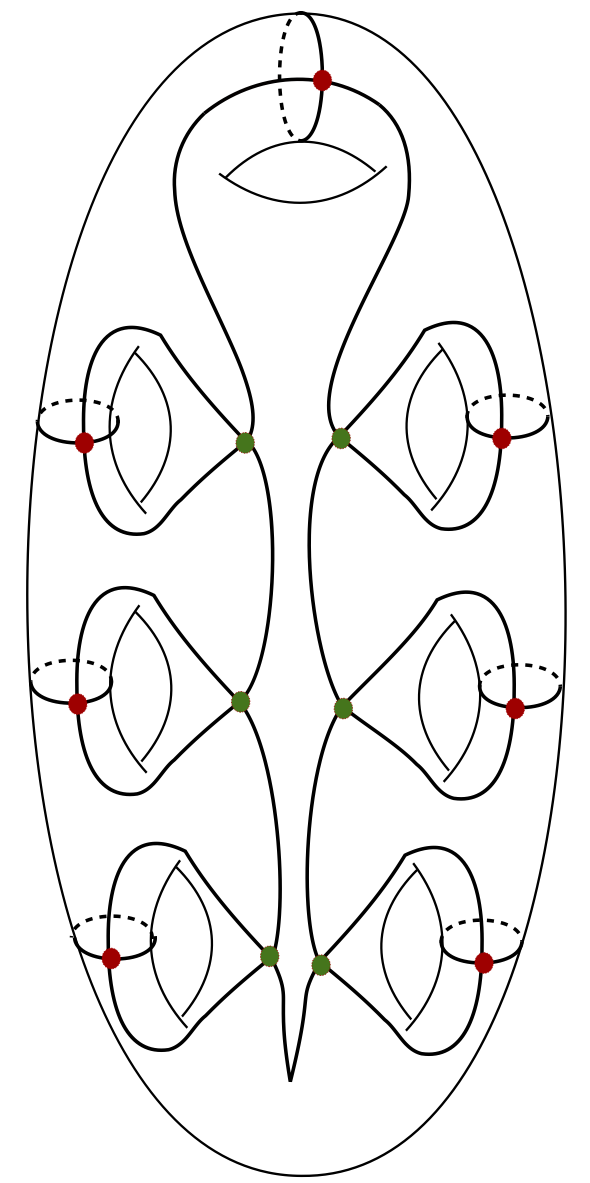}\hspace{3cm}
\includegraphics[scale=0.22]{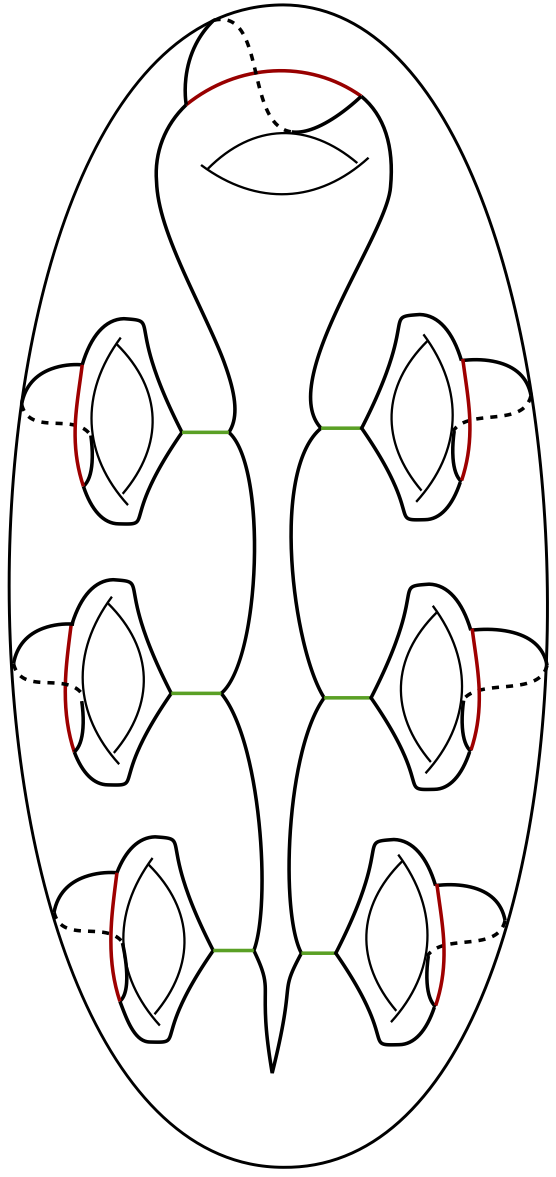}
\put(-43,18){\huge{$\longrightarrow$}}
\caption{}
\label{mitochondrie}
\end{center}
\end{figure}
\begin{lemma}\label{perfmatch2}
Let $G$ be a virtual unicellular collection. Then, there is a sequence of surgeries transforming $G$ into $CV_7$. 
\end{lemma}
\begin{proof}
Let $P$ be a perfect matching of $G$. By forgetting edges of $P$, one can see $(G,P)$ as a unicellular collection. By Theorem \ref{K4}, there is a sequence of surgeries (which does not touch edges in $P$) transforming $(G,P)$ into $(G_n,P_n)$ such that $G_n$ is obtained from $C_7$ by splitting its vertices. 
\begin{figure}[htbp]
\begin{center}
\includegraphics[scale=0.17]{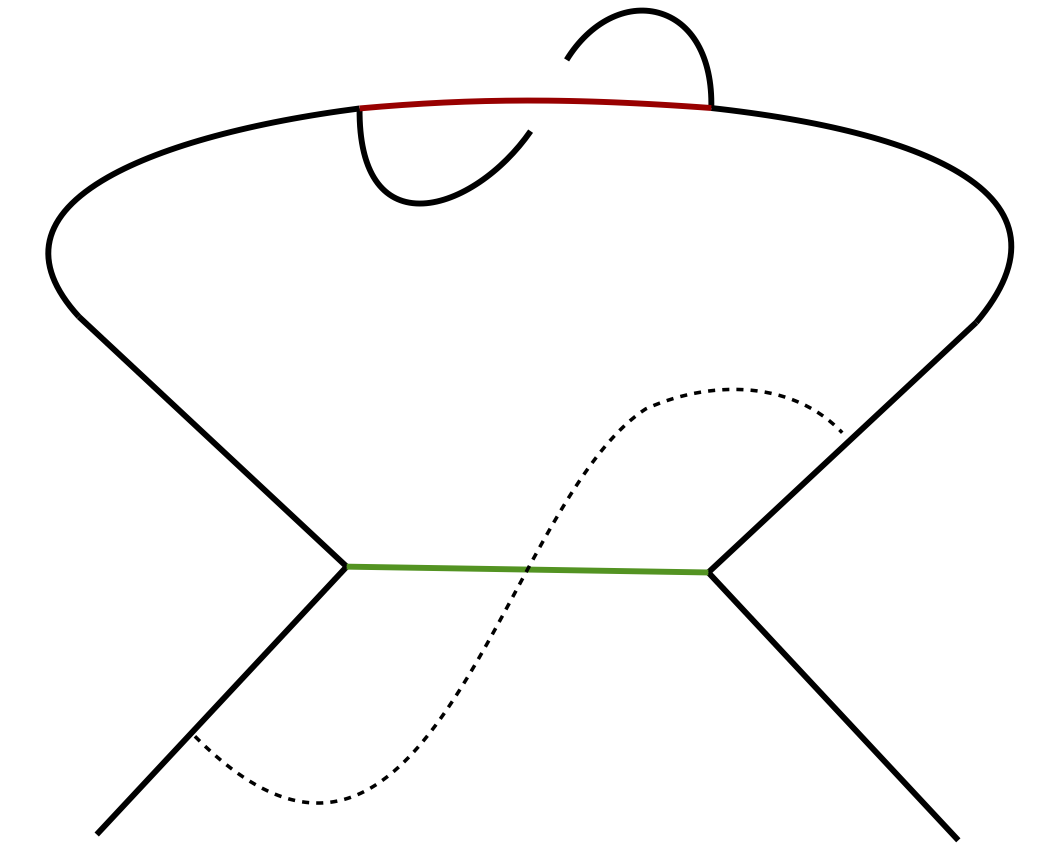}\hspace{3cm}
\includegraphics[scale=0.17]{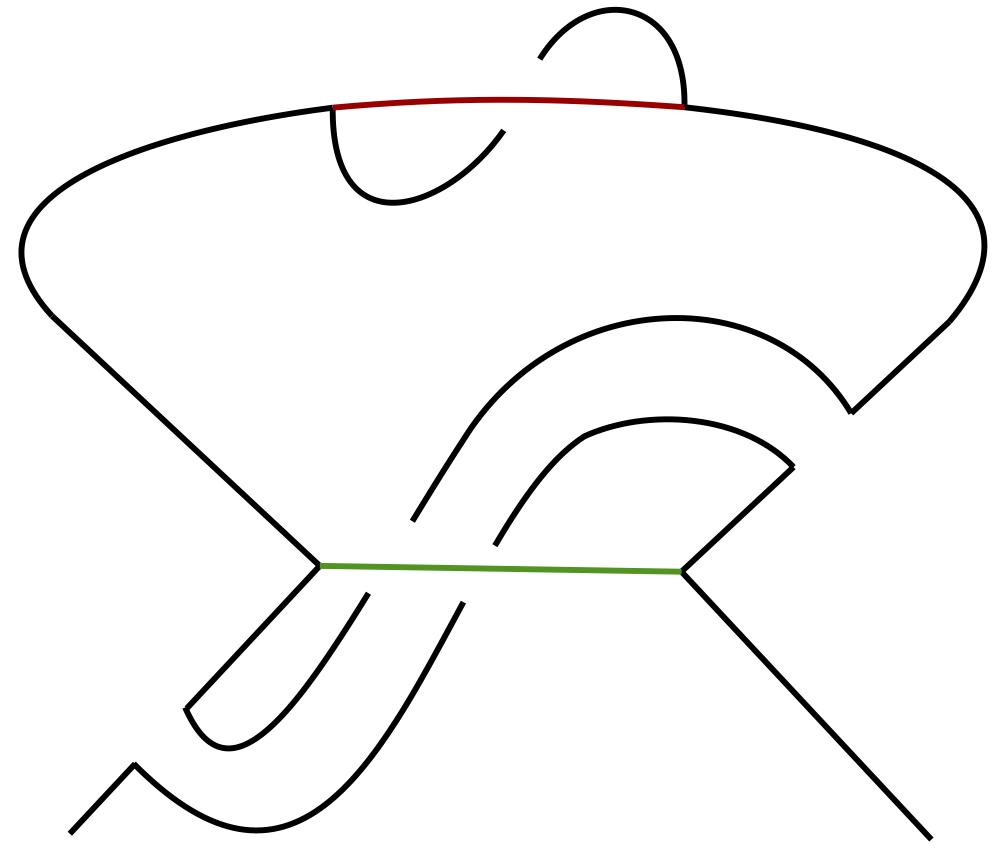}
\put(-51,7){\huge{$\longrightarrow$}}
\caption{A surgery which changes an edge into a  bridge.}
\label{..}
\end{center}
\end{figure}
The two ways on which one can split vertices of $C_7$ in red color lead to the same map. Also, the surgery depicted in Figure \ref{..} transforms an edge obtained by splitting a green vertex to a bridge edge, if it is not one. 

Therefore, there is sequence of surgeries from $G$ to $CV_7$ as claimed.   
\end{proof}
\begin{proof}[\textbf{Proof of Theorem \ref{K3}}]
Let $G$ be cubic unicellular map on $\sg_7$. By Lemma \ref{perfmatch1}, there is sequence of surgeries which transforms $G$ into $G_1$ such that $G_1$ admits a perfect matching. By Lemma \ref{perfmatch1}, there is sequence of surgeries which transforms $G_1$ into $CV_7$. So $K_{d_0,7}$ is connected. 
\end{proof}
\end{section}

\begin{section}{Surgeries and isotopy classes of homeomorphisms}\label{sec3}
This section start by the construction of explicit Dehn twists obtained by surgeries. Its ends with the proof of Main theorem. 

We recall that in this section, we assume that all unicellular maps have \textit{\textbf{false vertices}}\string: 2-valent vertices in each edges. 
 
Let $\mathrm{Homeo}^+(\sg_g)$  be the group of preserving orientation homeomorphism and $\mathrm{Homeo}^+_0(\sg_g)$ the subgroup of homeomorphism isotopic to identity. The mapping class group $\mathcal{MCG}(\sg_g):=\mathrm{Homeo}^+(\sg_g)/{\mathrm{Homeo}_0(\sg_g)}$ is the group of isotopy classes of preserving orientation homeomorphisms. Its well known (\cite{Dehn}, \cite{Lick}) that the mapping class group of $\sg_g$ is finitely generated by Dehn twist.  Humphries \cite{Humph} showed that $2g+1$ Dehn twists generated the mapping class group and $2g+1$ is the minimal number.    

Let $\widetilde{K}_{d,g}$ be a the topological surgery graph on $\widetilde{\mathcal{U}}_{d,g}$ and $K_{d,g}$ the combinatorial surgery graph. Let $G$ a fixed point in $K_{d,g}$, $\mathrm{Sym}(G)$ the group of homeomorphisms that fixe $G$, and $\gamma$ a closed path in $K_{d,g}$ based in $G$. 
A lift ~$\tilde{\gamma}$ of ~$\gamma$ in $\widetilde{K}_{d,g}$ gives a path from $\gamma(0)$ to $\gamma(1)$. There is a unique homeomorphism ~$\phi_{\gamma}$ (Up to $\mathrm{Sym}(G)$) such that $\gamma(1)=\phi_{\gamma}(\gamma(0))$. Therefore, the map 
 \begin{align*}
\mathcal{R}:\pi_1(K_{d_0,g}, G)&\longrightarrow \mathcal{MCG}(\sg_g)|\mathrm{Sym}(G)\\
                     \gamma &\longmapsto \phi_{\gamma},
\end{align*} 
 is well defined.
 \begin{definition}
 An homeomorphism $\phi$ is \textit{\textbf{$G$-surgery  compatible}} if $\phi=~\mathcal{R}(\gamma)$ for $\gamma\in\pi_1(K_{d_0,g},G)$. It is equivalent to say that $G$ and $\phi(G)$ are connected in the graph $\widetilde{K}_{d,g}$.     
 \end{definition}
 
 The subgroup of $G$-surgery compatible homeomorphisms $\mathcal{MCG}_G(\sg_g)$ is then generated by $\pi_1(K_{d_0,g}, G)$ and those generators come with relations between them.
 
\begin{paragraph}{Dehn twists as surgery compatible homeomorphisms\string:} Now, we show how we can obtained Dehn twists as $G$-surgery compatible homeomorphisms. 
 Let $x$ and $y$ be two intertwined edges in $G$. Then, if we take the arcs $\lambda_{x,y}$ (from $x$ to $y$) and $\lambda_{\bar{x},\bar{y}}$  (from $x$ to $y$) such that they end points are the same, we obtain an essential curve $\nu$ with one self-intersection. Let ~$\mathcal{P}(x,y)$ be the pair of pants generated by $\nu$; $\gamma$ the boundary component of $\mathcal{P}(x,y)$ which intersects $x$ and $y$ both, $\alpha$ the boundary component which intersects $x$, and $\beta$ the one which intersects $y$ (see Figure ~\ref{pant}).
\begin{figure}[htbp]
\begin{center}
\includegraphics[scale=0.25]{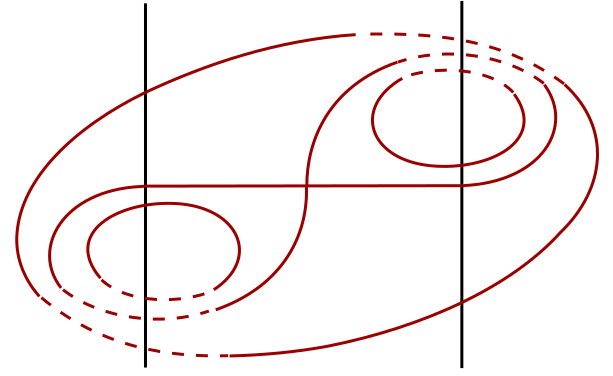}
\put(-37,18){\small{$\lambda_{x,y}$}}
\put(-28,10){\small{$\lambda_{\bar{x},\bar{y}}$}}
\put(-36,10){\small{$\alpha$}}
\put(-20,22){\small{$\beta$}}
\put(-55,13){\small{$\gamma$}}
\caption{Pair of pants generated by two intertwined edges $x$ and $y$.}
\label{pant}
\end{center}
\end{figure}
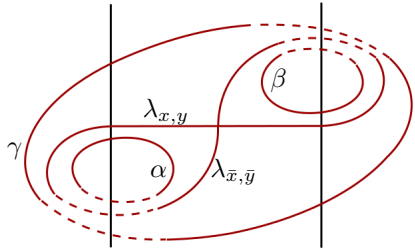
 
If $G':=\sigma_{x,y}(G)$ is the unicellular map obtained after a surgery on $G$ between $x$ and $y$, the oriented edges get transformed to two new oriented edges ---let us call them $x$ and $y$ too. By Lemma ~\ref{lemsurg}, $x$ and $y$ remain intertwined in $G'$. Moreover, if $w_1\bm{x}w_2\bm{\bar{x}}w_3\bm{y}w_4\bm{\bar{y}}$ is a word associated to $G$, then $w_3\bm{x}w_2\bm{\bar{x}}w_1\bm{y}w_4\bm{\bar{y}}$ is a word associated to $G'$ and $w_3\bm{x}w_4\bm{\bar{x}}w_1\bm{y}w_2\bm{\bar{y}}$ is a word associated to $\sigma_{\bar{x},\bar{y}}(G')$. The first and latest words are the same up to cyclic permutation and relabeling. Therefore $\sigma_{\bar{x},\bar{y}}(\sigma_{x,y}(G))$ defines a closed path based on $G$ in $K_{d,g}$ (See Figure \ref{involution} for the different steps of the path). 
\begin{figure}[htbp]
\begin{center}
\includegraphics[scale=0.14]{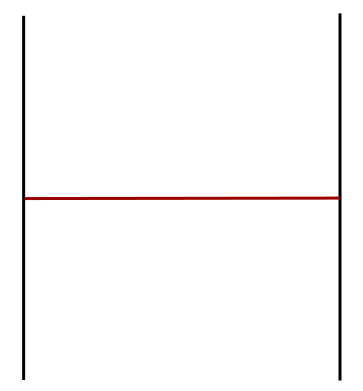}\hspace{1,5cm}
\includegraphics[scale=0.14]{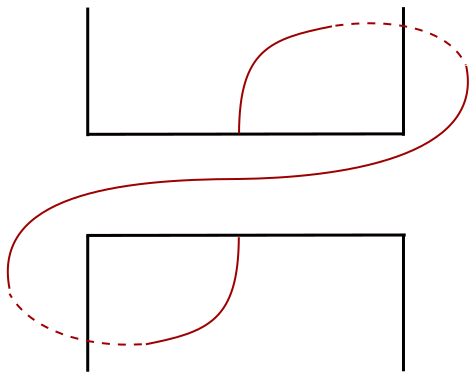}\hspace{1,5cm}
\includegraphics[scale=0.14]{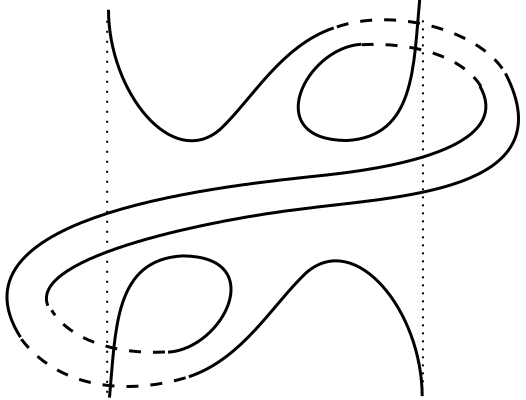}
\put(-80,6){\huge{$\rightarrow$}}
\put(-38,6){\huge{$\rightarrow$}}
\caption{}
\label{involution}
\end{center}
\end{figure}

Let us denote by $i_{x,y}$ this path and $\phi_{i_{x,y}}$ the homeomorphism defined by ~$i_{x,y}$. 
 \begin{lemma}\label{lemimport}
Let $i_{x,y}$ be the closed path in $K_{d,g}$ described above. Then,
 \[\phi_{i_{x,y}}=\tau^2_{\alpha}.\tau_{\gamma}.\tau^2_{\beta};\] 
 where $\tau_{\gamma}$, $\tau_{\alpha}$ and $\tau_{\beta}$ are Dehn twists along $\gamma$, $\alpha$ and ~$\beta$; respectively. 

 If $\omega$ is a simple curve in $\sg_g$ which intersects $G$ once, then $\tau^2_{\omega}$ is a $G$-surgery  compatible homeomorphism.  
 \end{lemma}
\begin{proof}
If $x$ and $y$ are intertwined, they defined a pair of pants and the sequence $G\longrightarrow\sigma_{x,y}(G)\longrightarrow\sigma_{\bar{x},\bar{y}}(\sigma_{x,y}(G))$ is depicted in Figure \ref{pantsuite} and the homeomorphism follows. 
\begin{figure}[htbp]
\begin{center}
\includegraphics[scale=0.16]{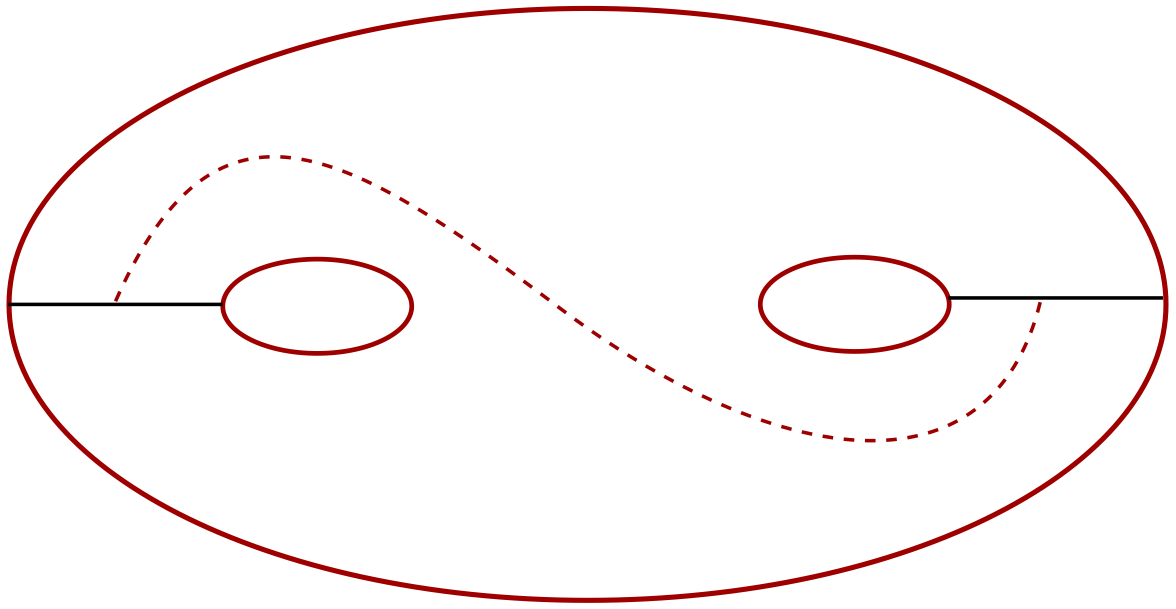}\hspace{0,4cm}
\includegraphics[scale=0.16]{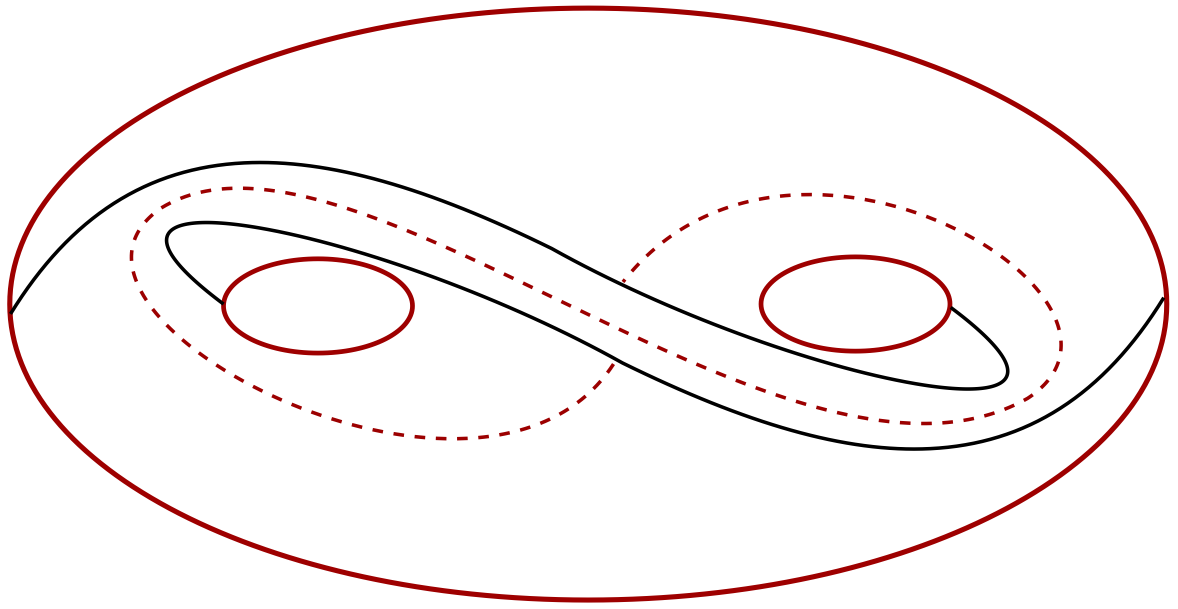}\hspace{0,4cm}
\includegraphics[scale=0.16]{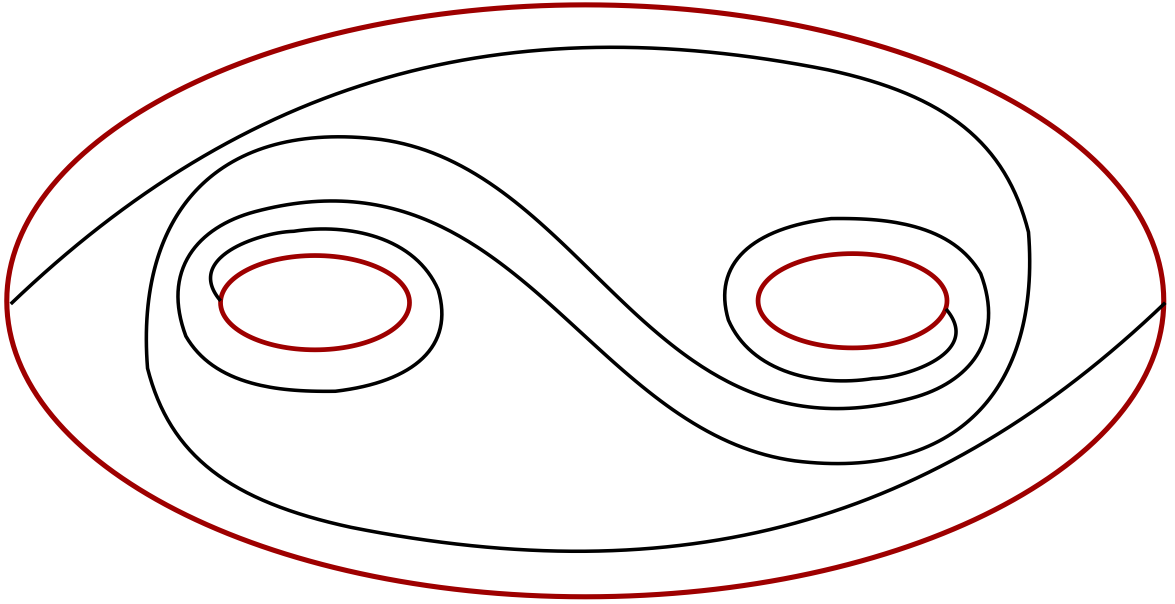}
\put(-74,7){\Large{$\rightarrow$}}
\put(-37,7){\Large{$\rightarrow$}}
\caption{}
\label{pantsuite}
\end{center}
\end{figure}

Let $\alpha$ be a simple curve that intersects $G$ once at $x$. If we consider the false vertex in $x$, the oriented edges $x$ breaks into two edges $x_1$ and $x_2$. The edges $x_1$ and $\bar{x}_2$ are intertwined and $\lambda_{x_1,\bar{x}_2}$ follows $\alpha$. The surgery on $G$ between $x_1$ and $\bar{x}_2$ gives the map $\tau^2_{\alpha}$ (See Figure ~\ref{square}).  

\begin{figure}[htbp]
\begin{center}
\includegraphics[scale=0.4]{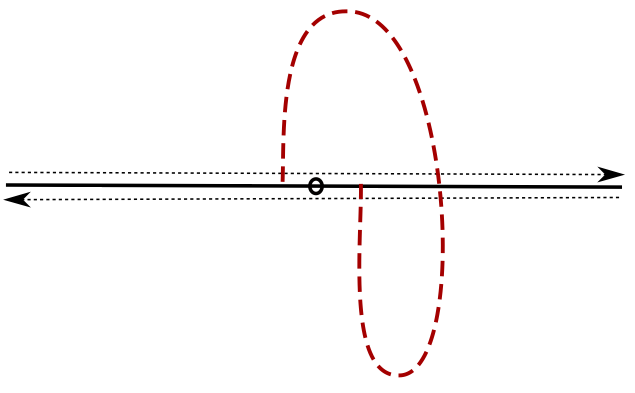}\hspace{2cm}
\includegraphics[scale=0.4]{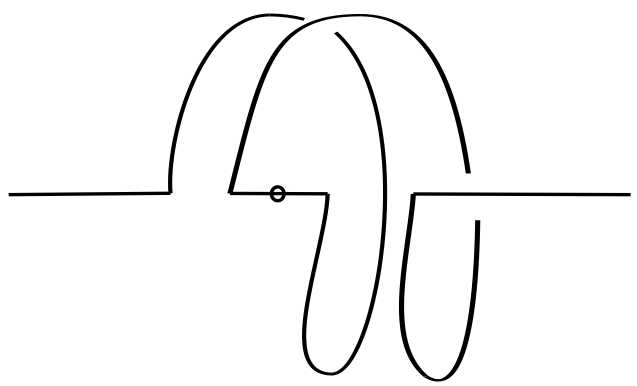}
\put(-58,12){\huge{$\rightarrow$}}
\put(-104,16){\small{$x_1$}}
\put(-104,10){\small{$\bar{x}_1$}}
\put(-74,16){\small{$x_2$}}
\put(-74,10){\small{$\bar{x}_2$}}
\put(-86,27){\small{$\lambda_{x,\bar{x}_2}$}}
\caption{}
\label{square}
\end{center}
\end{figure}

\end{proof}
\begin{corollary}\label{twist}
Let $x$ and $y$ be two intertwined edges of $G$ and $\gamma$ the unique simple closed curve which intersects $G$ twice, at $x$ and $y$. Then, $\tau_{\gamma}$ is a $G$-surgery  compatible homeomorphism.  
\end{corollary}
\begin{proof}
By Lemma \ref{lemimport}, $\tau^2_{\alpha}.\tau_{\gamma}.\tau^2_{\beta}$ is $G$-surgery  compatible. Since $\alpha$ and $\beta$ intersect $G$ once, $\tau^2_{\alpha}$  and $\tau^2_{\beta}$ are also $G$-surgery  compatible. We obtained $\tau_{\gamma}$ as a $G$-surgery  compatible homeomorphism by composition with $\tau^{-2}_{\alpha}$ and ~$\tau^{-2}_{\beta}$.  
\end{proof}
\begin{corollary}
For $d_0:=(3,3,...,3)$ and $g\geq2$, $\mathcal{MCG}_{d_0,g}(\sg_g)=\mathcal{MCG}(\sg_g)$. In other words, the whole mapping class group is $(d_0,g)$-surgery compatible for all $g\geq2$. 
\end{corollary}
\begin{proof}
By Corollary ~\ref{twist}, it is sufficient to find a set of simple closed curve ~$\gamma_{x,y}$ generated by intertwined edges, which is homeomorphic to Lickorish's generating set. The graph in dark in Figure ~\ref{humphgen} is a cubic unicellular map in ~$\sg_3$, and it has three bridges. 

\begin{figure}[htbp]
\begin{center}
\includegraphics[scale=0.26]{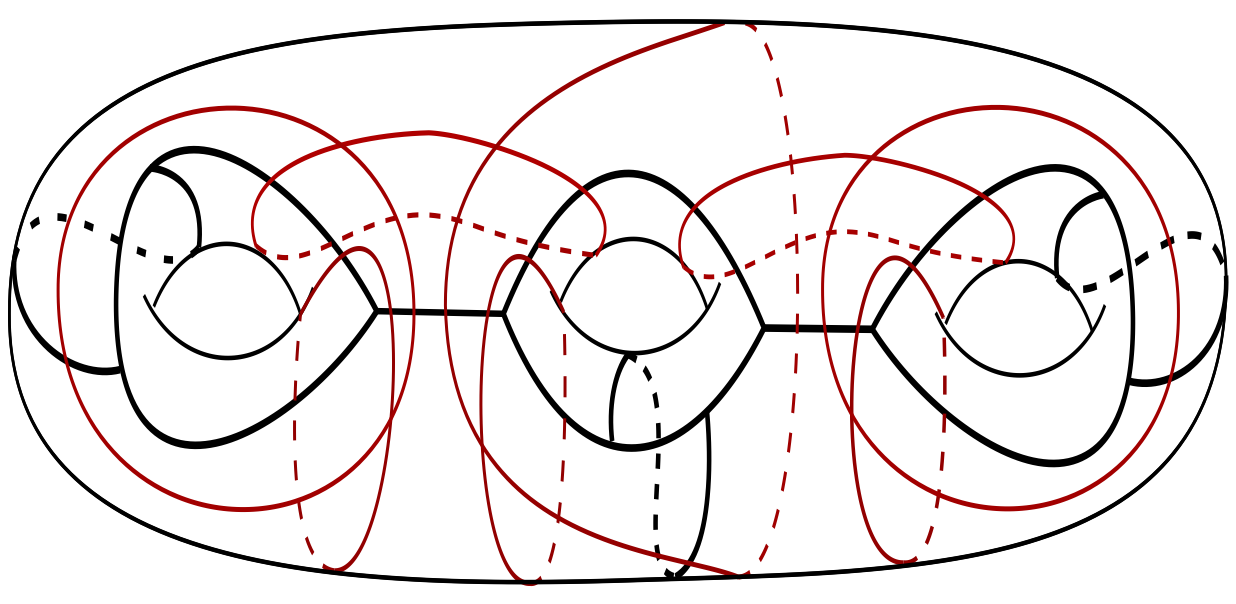}
\caption{Lickorish's generators  obtained by the simple curves generated by intertwined edges.}
\label{humphgen}
\end{center}
\end{figure}

A bridge edge is intertwined to all edges and two edges separate by a bridge are intertwined. The collection of curve in red is such that each curve intersects $G$ twice, at two intertwined edges. The right Dehn twists along these curves are $(d_0,3)$- surgery compatible and the collection is homeomorphic to the standard Lickorish's generators. So, ~$\mathcal{MCG}_{d_0,3}(\sg_3)=\mathcal{MCG}(\sg_3)$ and the proof is the same for all $g\geq2$.      
\end{proof}
\end{paragraph}

Here is the proof of Main theorem.
\begin{proof}[\textbf{Proof of Main theorem}]
Let $G_1$ and $G_2$ be two points in $\widetilde{\mathcal{U}}_{d_0,g}$ ($g\geq2$). Since $K_{d_0,g}$ is connected (Theorem ~\ref{K3}), there is a path from $G_1$ to $\phi(G_2)$. Since all homeomorphisms are $(d_0,g)$-surgery compatible, there is a path from $\phi(G_2)$ to $G_2$. So, $\widetilde{K}_{d_0,g}$ is connected.  

For the quasi-isometry between $\mathcal{MCG}(\sg_g)$ and $\widetilde{K}_{d_0,g}$, we check \v{S}varc-Milnor lemma's conditions. 

The graph $\tilde{K}_{d_0,g}$ is a proper length space (edges have length equal to one). The action of $\mathcal{MCG}(\sg_g)$ is co-compact since the graph $K_{d_0,g}$ is compact. For every $G\in\widetilde{K}_{d_0,g}$, $|\mathrm{Fix}(G)\leq\infty|$ since $G$ as finite number of edges. Therefore, the action is properly discontinuous. So, $\mathcal{MCG}(\sg_g)$ is quasi-isometric to $\widetilde{K}_{d_0,g}$.  
\end{proof}

\begin{remark} In general, $\mathcal{MCG}_G(\sg_g)$ is quasi-isometric to $C(G)$, where $C(G)$ is the connected component of $\widetilde{K}_{d,g}$ containing $G$. 
\end{remark}

One could think that the equality $\mathcal{MCG}_{d_0,g}(\sg_g)=\mathcal{MCG}(\sg_g)$ for $g\geq2$ is due to the fact that cubic unicellular maps have the maximal number of edges. The unicellular map $G$ in Figure \ref{humph2} is of degree partition $d=(5,5)$ and we obtained the Lickorish's generators with the simple closed curves generated by intertwined edges in $G$. It implies that each connected component of $\widetilde{K}_{d,2}$ is stable under the action of the mapping class group. But, by tracking  the Humphries generators, one can generate the whole mapping class group with $G\in\widetilde{\mathcal{U}}_{d,g}$ for $d=(5,5,6,6...,6)$. To see this, we start with the analogue of the map $G$ depicted in Figure ~\ref{humphgen}. It gives us the $3g-1$ Lickorish's generators. We remove $g-2$ curves form that generating set so that we obtained the Humphries generators. We then collapse all edges which are disjoint from the remaining generating set.     

\begin{figure}[htbp]
\begin{center}
\includegraphics[scale=0.25]{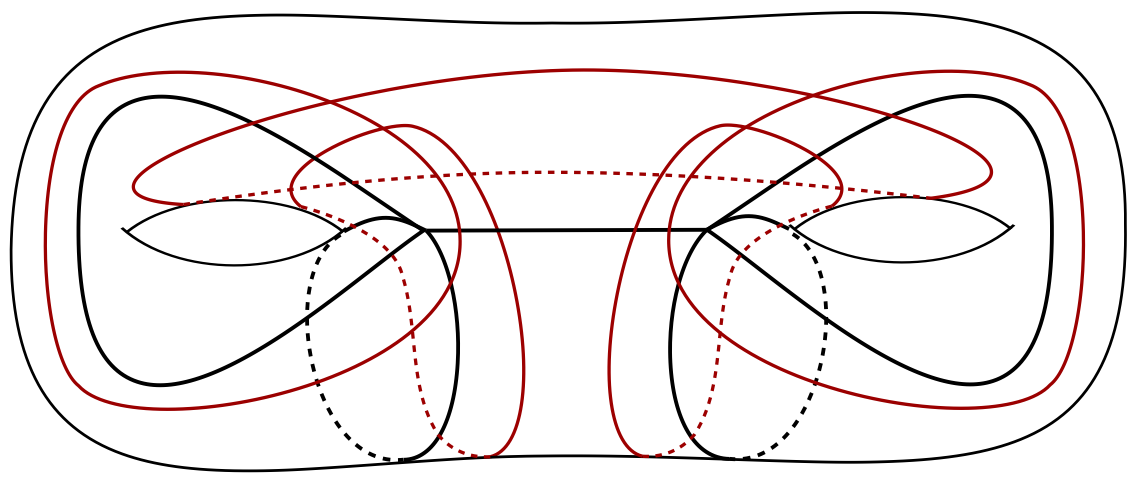}\hspace{1,5cm}
\caption{Lickorish's generators obtained form intertwined edges.}
\label{humph2}
\end{center}
\end{figure}

\begin{paragraph}{Surgery invariant\string:} Now, we define a surgery invariant for the proof of Theorem \ref{inv}.

\noindent Let $G$ be a unicellular map, $\mathcal{S}_g$ the set of all simple closed curves in $\sg_g$ and ~$I_{G}$ the map associated to $G$ defined by\string:

 \begin{align*}
I_{G}:\mathcal{S}_g&\longrightarrow \z/2\z\\
                     \alpha &\longmapsto i(G,\alpha)\mod 2;
\end{align*}

where $i(G,\alpha)$ is the geometric intersection number\string: the minimal number of time $\alpha$ intersects $G$ in its isotopy class. 

\begin{Prop}[\textit{Invariant}]
Let $G$ be a unicellular map whose vertices have even degree. If $G'$ is obtained by a surgery on $G$, then 
\[I_{G'}=I_{G}.\] 
\end{Prop}
\begin{proof}
Let $\alpha$ be a simple curve in $\sg_g$ and $i(G,\alpha):=k$. Let $x$ and ~$y$ be two intertwined edges of $G$ and $\lambda_{x,y}$ the arc from $x$ to $y$. It follows that after the surgery, $$|G'\cap\alpha|=k+2n;$$ where $n$ is the number of time $\lambda_{x,y}$ intersects $\alpha$. Since $G'$ has even degree partition, $i(G',\alpha)$ and $k+2n$ have the same party. 

So, $i(G,\alpha)=i(G',\alpha)\mod 2$. 
\end{proof}

With the proposition above, we provide many Dehn twists which are not $G$-surgery  compatible for even degree partitions.
\begin{corollary}\label{noncomp}
Let $G$ be a unicellular map with even degree partition and ~$\alpha$ be a simple curve such that $i(G,\alpha)$ is odd. Then, the Dehn twist $\tau_{\alpha}$ is not $G$-surgery  compatible; that is $\mathcal{MCG}_G(\sg_g)$ is a proper subgroup of ~$\mathcal{MCG}(\sg_g)$. 
\end{corollary}
\begin{proof}
Since $i(G,\alpha)$ is odd, then $\alpha$ is non separating. Therefore, there is a a simple curve $\beta$ which intersects $\alpha$ once. But, $i(\tau_{\alpha}(G),\beta)=i(G, \beta)+i(G, \alpha)$. It follows that $i(G,\beta)$ and $i(\tau_{\alpha}(G),\beta)$ have different parity. Hence, $\tau_{\alpha}$ is not $G$-surgery  compatible.  
\end{proof}

Given a unicellular map $G$ with even degree partition, The map $I_{G}$ is completely determined by its valued on a (symplectic) basis of  the first homology group of $\sg_g$. Therefore, $I_{G}$ is determined by a vector of size $2g$ with coordinates ~$0$ or ~$1$. The invariant $I$ defines $3g$ different classes.  

Figure \ref{diff} shows two different unicellular collections $G_1$ and $G_2$ with different associated maps\string: $I_{G_1}=(1,1,1,1)$ and $I_{G_2}=(1,0,1,1)$. So, they are in different connected components of $\widetilde{K}_{(4,4,4),2}$. 
\begin{figure}[htbp]
\begin{center}
\includegraphics[scale=0.22]{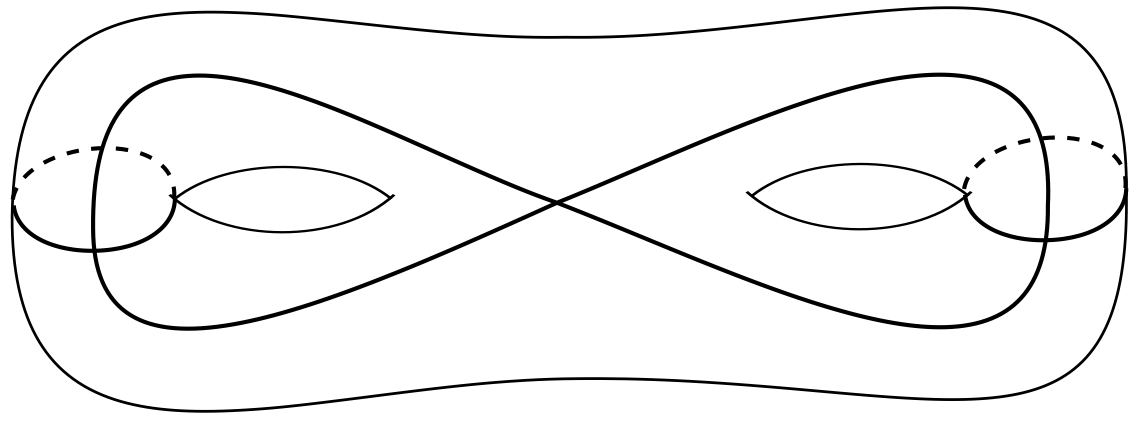}\hspace{1cm}
\includegraphics[scale=0.22]{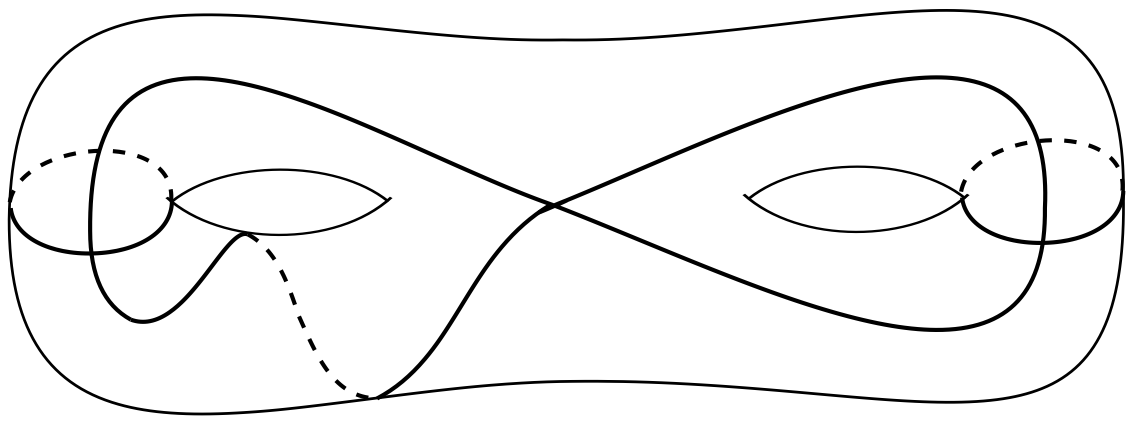}
\caption{Two unicellular collections in different connected components of $\tilde{K}_{(4,4,4),2}$.}
\label{diff}
\end{center}
\end{figure}

There is another interesting subgroup associated to a unicellular map of even degree partition $G$\string: \[\mathcal{MCG}_{I_G}(\sg_g)=\{\phi\in\mathcal{MCG}(\sg_g), I_{G}=I_{\phi(G)}\}.\]
The group $\mathcal{MCG}_G(\sg_g)$ is a subgroup of $\mathcal{MCG}_{I_G}(\sg_g)$.

Now we give the proof of Theorem \ref{inv}.
\begin{proof}[\textbf{Proof of theorem 3}] 

By Corollary \ref{noncomp}, Dehn twist along simple curve which intersects $G$ ---with even degree partion---, are not $G$-surgery compatible. It follows that $\mathcal{MCG}_G(\sg_g)$ is a proper subgroup.  
\end{proof}

The graph $K_{d_1,g}$ is disconnected but each connected component contains a $\mathcal{MCG}_{d_1,g}(\sg_g)$ orbit of any combinatorial map $[G]\in K_{d_1,g}$ since $K_{d_1,g}$ is connected. Theorem ~\ref{inv} implies that the question of connectedness of combinatoric surgery graphs on unicellular maps of even partition is purely combinatoric. 
\end{paragraph}

\begin{figure}[htbp]
\begin{center}
\includegraphics[scale=0.17]{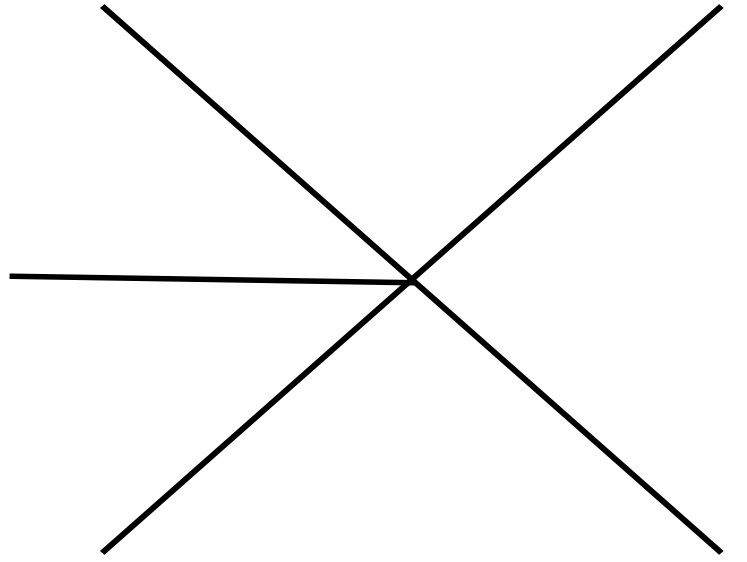}\hspace{3cm}
\includegraphics[scale=0.17]{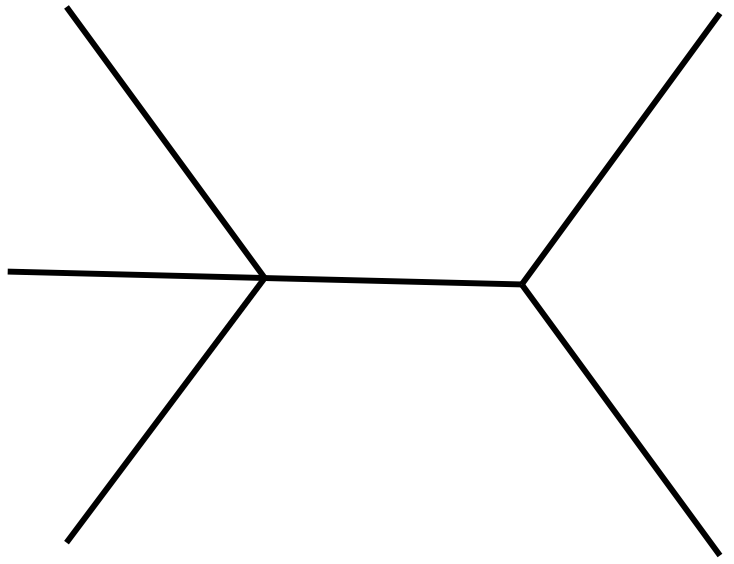}
\put(-43,7){\Large{$\longrightarrow$}}
\caption{Splitting a vertex of degree five into two vertices.}
\label{split}
\end{center}
\end{figure}
 
\begin{paragraph}{Filtration of $\mathcal{MCG}(\sg_g)$ by splitting vertex sequence\string:} We added more to the ubiquity of unicellular maps by showing that  a unicellular map $G$ parametrizes a subgroup of mapping class group\string: the subgroup $\mathcal{MCG}_G(\sg_g)$. The splitting operation on vertices (see Figure \ref{split}) gives a natural inclusion map. In fact, if ~$v$ is a vertex of a unicellular map $G$ with degree greater than $3$, one can split ~$v$ into two to obtained a new map $G'$. A surgery on $G$ has an equivalent in $G'$ by just forgetting new edges created by the splitting operation. This shows that $G$-surgery compatible homeomorphism, defined by a loop of surgeries in $G$, is also a $G'$-surgery compatible homeomorphism\string:
\[G\overset{split}{\longrightarrow} G'\quad\implies\quad \mathcal{MCG}_G(\sg_g) \hookrightarrow \mathcal{MCG}_{G'}(\sg_g).\]

We then defined a filtration of $\mathcal{MCG}(\sg_g)$ as follow\string: we start with a unicellular map $G$ with only one vertex and we split step by step vertices till we obtained a cubic unicellular map. To this sequence $G_0=G\rightarrow...\rightarrow G_n$ follows a filtration\string:
\[\mathcal{MCG}_G(\sg_g)\hookrightarrow\mathcal...\hookrightarrow \mathcal{MCG}_{G_{n-1}}(\sg_g)\hookrightarrow \mathcal{MCG}(\sg_g)\]  
\end{paragraph} 
\end{section}
\begin{section}{Two examples}\label{sec4}
\begin{paragraph}{The case of torus\string:} 
 Let us consider $\widetilde{\mathcal{U}}_{(4),1}$\string: the set of unicellular maps on ~$\mathbb{T}^2$ with one vertex of degree $4$. The geometric intersection between $G\in\widetilde{\mathcal{U}}_{(4),1}$ and the two curves of the symplectic basis can not be even at the same time. Therefore, the invariant $I$ defines three different classes\string: 
\[(0,1);\quad (1,0);\quad (1,1).\]
Let $G_0$ be the unicellular maps on $\mathbb{T}^2$ made by the meridian curve $\alpha$ and the longitude $\beta$. Let $\phi_0$ be the homeomorphism of $\mathbb{T}^2$ which maps $\alpha$ to $\beta$ and vice versa; $\phi_0^2=\mathrm{Id}$. 
To $G\in\widetilde{\mathcal{U}}_{(4),1}$, we associate the following word\string:
\[x_1x_2y_1y_2\bar{x}_2\bar{x}_1\bar{y}_2\bar{y_1}.\]    

Then, oriented edges which are not in the same curve are not intertwined. Therefore, $\mathcal{MCG}_G(\mathbb{T}^2)=\langle \tau^2_{\alpha},\tau^2_{\beta}\rangle$. So, $\phi_0$ is not $G$-surgery compatible since $\mathcal{MCG}_G(\mathbb{T}^2)$ is free; even though $\phi_0\in \mathcal{MCG}_{I_{G_0}}(\mathbb{T}^2)$. It means that in this case $\mathcal{MCG}_{d_0,1}(\mathbb{T}^2)$ is a proper subgroup of $\mathcal{MCG}_{I_{G_0}}(\mathbb{T}^2)$, but up to $\phi_0$, $\mathcal{MCG}_{I_{G_0}}(\mathbb{T}^2)$ is equal to $\mathcal{MCG}_{d_0,1}(\mathbb{T}^2)$.  
\end{paragraph}
\begin{paragraph}{The partition $d=(4g)$ and $d=(2g+1, 2g+1)$\string:} We chose these two partition because they are examples for which the combinatorial surgery graph is not connected. Let $[G]$ be combinatorial map on $\sg_g$ defined by $W_{G}=x_1x_2...x_g\bar{x}_1\bar{x}_2...\bar{x}_g$. The map $G$ has one vertex and there are not pair of intertwined edges. So the only elementary $G$-surgery compatible homeomorphism are square of Dehn twists along simple curves $\alpha_{x_i}$ supported by the edges, and $[G]$ is an isolated point in $K_{(4g),g}$. Moreover, a reduced sequence of surgeries (two consecutive surgeries do not cancel each other) acts non trivially on $G$. In fact, at the $k$-th step of the sequence of surgeries, the arc along which the surgery is performed follows the meanders of the previous one since $x_i< x_j<\bar{x}_i<\bar{x}_j$ for all $i<j$; that is the map $G$ gets more complicated.  So,
\[\mathcal{MCG}_G(\sg_g)=\langle \tau^2_{\alpha_{x_1}},...,\tau^2_{\alpha_{x_g}}\rangle\simeq F_g.\]

Let $[G']$ be the map defined by the word $W_{G'}=x_1x_2\bar{x}_1\bar{x}_2...x_{g-1}x_g\bar{x}_{g-1}\bar{x}_g$. The map $[G']$ has also one vertex with many intertwined oriented edges. The simple curves $\alpha_1$ and  $\alpha_2$ generated by the intertwined pairs $\{x_1,x_4\}$ and $\{x_2,x_3\}$ are disjoint. So, $\langle \tau_{\alpha_1}, \tau_{\alpha_2}\rangle$ is a subgroup of $\mathcal{MCG}_{G'}(\sg_g)$ isomorphic to $\mathbb{Z}^2$. Therefore, $G$ and $G'$ parametrize different subgroups of $\mathcal{MCG}(\sg_g)$.

For $d=(2g+1,2g+1)$, the same thing happens. The graph $K_{d,g}$ is not connected ---the combinatorial map $[G]$ defined by  $x_1x_2...x_{g+1}\bar{x}_1\bar{x}_2...\bar{x}_{g+1}$ has two vertices and is isolated in $K_{d,g}$. Moreover, 
\[\mathcal{MCG}_G(\sg_g)=\langle \tau^2_{\alpha_{x_1}},...,\tau^2_{\alpha_{x_{g+1}}}\rangle=F_{g+1}.\] 

The map $G'$ defined by $x_1x_2\bar{x}_1\bar{x}_2...x_{g-1}x_g\bar{x}_{g-1}\bar{x}_g$ defines a proper subgroup different to $\mathcal{MCG}_G(\sg_g)$ except that for $g=2$, $\mathcal{MCG}_{G'}=\mathcal{MCG}(\sg_g)$ ---the simple curves generated by intertwined edge gives Lickorish's generetors. 
\end{paragraph}
\end{section}      
\vspace{0,5cm}    
Institut Fourier, Université Grenoble Alpes.\\
\textit{email\string: karimka02@hotmail.fr}
\end{document}